\documentclass[11pt]{amsart}

\usepackage[margin=1.12in]{geometry}
\usepackage{amsmath,amssymb,amsthm,mathtools}
\usepackage[hidelinks]{hyperref}
\usepackage{microtype}

\allowdisplaybreaks

\newtheorem{theorem}{Theorem}[section]
\newtheorem{lemma}[theorem]{Lemma}
\newtheorem{proposition}[theorem]{Proposition}
\newtheorem{corollary}[theorem]{Corollary}
\newtheorem{remark}[theorem]{Remark}

\numberwithin{equation}{section}

\newcommand{\N}{\mathbb N}
\newcommand{\R}{\mathbb R}
\newcommand{\Z}{\mathbb Z}
\newcommand{\one}{\mathbf 1}
\newcommand{\eps}{\varepsilon}
\DeclareMathOperator{\sgn}{sgn}

\title[Geometric block exponents and mixed sumsets]
{Geometric block exponents and a uniform mixed-alphabet sumset inequality}

\author{Johannes Hosle}
\address{Department of Mathematics, Massachusetts Institute of Technology, Cambridge, MA 02139, USA}
\email{jhosle@mit.edu}

\author{Paata Ivanisvili}
\address{Department of Mathematics, University of California, Irvine, Irvine, CA 92697, USA}
\email{pivanisv@uci.edu}

\date{}

\hypersetup{
  pdftitle={Geometric block exponents and a uniform mixed-alphabet sumset inequality},
  pdfauthor={Johannes Hosle and Paata Ivanisvili}
}

\begin{document}

\begin{abstract}
We study sharp exponents in inequalities for pairs of finite geometric
blocks.  We characterize exactly when the endpoint $t=1$ determines the
optimal exponent and compute the exponent that is uniform in the length of
one block.  For a two-term first block the answer is
$p_0=\log 4/\log 6$.  This yields a uniform two-slice max-convolution
inequality and, for every $m,d\ge1$, the dimension-free mixed-alphabet
sumset bound
\[
        |A+B|\ge (|A||B|)^{p_0},
        \qquad
        A\subset\{0,1\}^d,\quad
        B\subset\{0,1,\ldots,m\}^d.
\]
For every $m\ge2$, the exponent $p_0$ is best possible; for $m=1$, a
larger exponent is available.
\end{abstract}

\maketitle

\section{Introduction and main results}\label{sec:intro}

Throughout, $\N=\{1,2,3,\ldots\}$.  For $n\ge0$ and $0\le t\le1$, put
\[
        G_n(t):=1+t+\cdots+t^n.
\]
For $r,s\in\N$, call an exponent $p>0$ \emph{admissible} for $(r,s)$ if
\begin{equation}\label{eq:block-def-intro}
        G_{r+s}(t^p)\ge G_r(t)^pG_s(t)^p
        \qquad(0\le t\le1),
\end{equation}
and declare $p=0$ admissible as well.  Let $p_{r,s}$ denote the largest
admissible exponent; its existence is proved in
Lemma~\ref{lem:basic-p}.  Testing \eqref{eq:block-def-intro} at $t=1$ gives
\[
        p_{r,s}\le
        a_{r,s}:=
        \frac{\log(r+s+1)}{\log(r+1)+\log(s+1)}.
\]
A second distinguished quantity is
\[
        b_{r,s}:=
        \frac{r(r+2)+s(s+2)}{(r+s)(r+s+2)};
\]
it is the obstruction obtained from the quadratic term in the expansion at
$t=1$.

Our first result determines exactly when the endpoint upper bound is sharp.

\begin{theorem}[Endpoint-active block exponents]\label{thm:endpoint-active}
For every $r,s\in\N$,
\[
        p_{r,s}=a_{r,s}
        \quad\Longleftrightarrow\quad
        a_{r,s}\ge b_{r,s}.
\]
\end{theorem}

The next theorem determines the optimal exponent when the first block length
is fixed and the second is arbitrary.

\begin{theorem}[Fixed-head constant]\label{thm:fixed-head}
For fixed $r\in\N$, define
\[
        C_r:=\inf_{s\ge1}p_{r,s}.
\]
Then
\begin{equation}\label{eq:fixed-head-identity-intro}
        C_r=\min_{s\ge r}a_{r,s}.
\end{equation}
There is a unique $\sigma_r\in(r,\infty)$ satisfying
\begin{equation}\label{eq:sigma-intro}
 (\sigma_r+1)\log\bigl((r+1)(\sigma_r+1)\bigr)
 =
 (r+\sigma_r+1)\log(r+\sigma_r+1).
\end{equation}
If $k_r=\lfloor\sigma_r\rfloor$, then
\begin{equation}\label{eq:fixed-head-two-intro}
        C_r=\min\{a_{r,k_r},a_{r,k_r+1}\}.
\end{equation}
Moreover,
\begin{equation}\label{eq:real-min-intro}
 \inf_{\sigma\ge r}
 \frac{\log(r+\sigma+1)}{\log(r+1)+\log(\sigma+1)}
 =\frac{\sigma_r+1}{r+\sigma_r+1}.
\end{equation}
\end{theorem}

For the applications, the case $r=1$ is decisive.

\begin{corollary}\label{cor:C1}
One has
\[
        C_1=p_{1,2}=\frac{\log4}{\log6}=:p_0.
\]
Consequently, for every integer $L\ge0$ and every $0\le t\le1$,
\begin{equation}\label{eq:geometric-tail-intro}
        G_{L+1}(t^{p_0})
        \ge \bigl((1+t)G_L(t)\bigr)^{p_0}.
\end{equation}
\end{corollary}

We also determine $\lim_{s \to \infty} p_{1, s}$.

\begin{proposition}\label{prop:p1s-limit}
We have
\begin{equation*}
    \lim_{s\to\infty}p_{1,s}
    =
    \frac{\log2}{\log(1+\sqrt2)}.
\end{equation*}
Equivalently, the optimal exponent in
\begin{equation}\label{eq:p1s-finite-ineq}
    G_{s+1}(t^p)\ge \big((1+t)G_s(t)\big)^p,
    \qquad 0\le t\le1,
\end{equation}
converges to $\log2/\log(1+\sqrt2)$ as $s\to\infty$.
\end{proposition}

For nonnegative sequences $x=(x_0,x_1)$ and
$y=(y_0,\ldots,y_m)$, write
\[
        (x\star_{\max}y)_k:=\max_{i+j=k}x_i y_j,
\]
where the maximum is restricted to the available indices.  The fixed-head
constant gives the following uniform inequality for arbitrary nonincreasing
sequences.

\begin{theorem}[Uniform two-slice max-convolution]\label{thm:two-slice}
Let $m\ge1$, let $x_0\ge x_1\ge0$, and let
$y_0\ge y_1\ge\cdots\ge y_m\ge0$.  Then
\begin{equation}\label{eq:two-slice}
\sum_{k=0}^{m+1}
\left(\max_{\substack{i+j=k\\0\le i\le1\\0\le j\le m}}x_i y_j\right)^{p_0}
\ge
(x_0+x_1)^{p_0}
\left(\sum_{j=0}^m y_j\right)^{p_0}.
\end{equation}
For every $m\ge2$, no larger exponent can hold uniformly.
\end{theorem}

\begin{remark}\label{rem:m-one}
For $m=1$, the exponent $p_0$ is not optimal.  The sharp exponent in the
two-slice inequality is
\[
        q_1=\frac{\log3}{\log4}>p_0;
\]
see \cite[Theorem~8]{Hosle2026}.  Thus the word ``sharp'' in the applications
below always refers to the range $m\ge2$.
\end{remark}

The sumset consequence is stated for Minkowski sums in $\Z^d$.

\begin{corollary}[Mixed binary--$m$-ary sumsets]\label{cor:sumset}
Let $m,d\ge1$.  If
\[
        A\subset\{0,1\}^d,
        \qquad
        B\subset\{0,1,\ldots,m\}^d,
\]
then
\begin{equation}\label{eq:mixed-sumset}
        |A+B|\ge (|A||B|)^{p_0}.
\end{equation}
For every $m\ge2$, the exponent $p_0$ is best possible.
\end{corollary}

Sharpness for $m\ge2$ is proved inside the proofs of
Theorem~\ref{thm:two-slice} and Corollary~\ref{cor:sumset}, using the same
three-term product example.

The paper is organized as follows.  Section~\ref{sec:history} records the
historical background.  Section~\ref{sec:prelim} gives the basic properties
of $p_{r,s}$ and states the geometric-block estimate proved by the first
author.  Sections~\ref{sec:endpoint} and \ref{sec:fixed-head} prove the two
general theorems.  Section~\ref{sec:applications} derives the max-convolution
and sumset consequences from Corollary~\ref{cor:C1}.

\section{Historical remarks}\label{sec:history}

Bourgain, Dilworth, Ford, Konyagin, and Kutzarova
\cite{BDFKK2011} introduced the product-set sumset problem considered here
as part of their construction of explicit restricted-isometry matrices.
For the full alphabet
$\mathcal C_m^d=\{0,1,\ldots,m\}^d$, the full-box example gives the upper
bound
\[
        q_m:=\frac{\log(2m+1)}{2\log(m+1)},
\]
and they suggested that this exponent should be optimal.  They also reduced
the corresponding dimension-free sumset inequality to a one-dimensional
max-convolution inequality for nonincreasing sequences.

Becker, Ivanisvili, Krachun, and Madrid \cite{BIKM2025} proved the sharp
full-alphabet result for $m=2$.  Their argument introduced a max-tie analysis
that reduces the problem to structured configurations, with geometric blocks
as the final model case.  The first author subsequently proved the
geometric-block inequality with exponent $q_m$ for every $m$ and proved the
max-convolution inequality with the same exponent when one sequence has only
two nonzero terms \cite[Theorems~1 and~8]{Hosle2026}.  The reduction in
Section~\ref{sec:two-slice} to adjacent ratios in $\{1,t\}$ and the ensuing
geometric-majorization comparison are adapted from the proof of
\cite[Theorem~8]{Hosle2026}.  The new input here is
Corollary~\ref{cor:C1}: it replaces the length-dependent exponent $q_m$ by
the uniform exponent $p_0$ and makes the conclusion sharp for every $m\ge2$.
In that range $p_0$ is strictly larger than $q_m$.

The first author's geometric-block proof is based on a comparison theorem
for Stolarsky means.  Stolarsky introduced these means in
\cite{Stolarsky1975}; their comparison theory was developed by Leach and
Sholander \cite{LeachSholander1983} and by P\'ales \cite{Pales1988}.  The
proof of Theorem~\ref{thm:endpoint-active} uses a different interpolation
mechanism: a determinantal sign rule for an extended Chebyshev system.
A complementary mixed-alphabet estimate appears in work of Gowers and Karam
\cite[Proposition~3.1]{GowersKaram2022}: if
$A\subset\{0,1\}^d$ and $B\subset\{0,1,\ldots,m\}^d$ have relative
densities $\alpha$ and $\beta$, respectively, then $A+B$ has relative
density at least $\alpha\beta$ in $\{0,1,\ldots,m+1\}^d$.  That estimate is
particularly effective for dense sets, while Corollary~\ref{cor:sumset} is a
dimension-free power bound that remains effective for sparse sets.

\section{Preliminaries on geometric blocks}\label{sec:prelim}

We first justify the use of the phrase ``largest exponent.''

\begin{lemma}[Basic properties of admissible exponents]\label{lem:basic-p}
For every $r,s\in\N$, the set of admissible exponents is a closed
interval $[0,p_{r,s}]$, and
\[
        \frac12\le p_{r,s}\le a_{r,s}<1.
\]
In particular, $p_{r,s}$ is attained.  Moreover, $p_{r,s}=p_{s,r}$.
\end{lemma}

\begin{proof}
Symmetry is immediate.  We first show that $p=1/2$ is admissible.  Put
$u=t^{1/2}$.  Since $0\le u\le1$,
\[
 G_{r+s}(u)\ge G_r(u)\ge G_r(u^2)=G_r(t)
\]
and similarly $G_{r+s}(u)\ge G_s(t)$.  Therefore
\[
 G_{r+s}(t^{1/2})\ge \max\{G_r(t),G_s(t)\}
 \ge \bigl(G_r(t)G_s(t)\bigr)^{1/2}.
\]

Testing an admissible exponent at $t=1$ gives
\[
 r+s+1\ge ((r+1)(s+1))^p,
\]
so $p\le a_{r,s}$.  Since
\[
 r+s+1<(r+1)(s+1),
\]
we have $a_{r,s}<1$.

We next prove downward closure.  Suppose that $p>0$ is admissible and
$0<q\le p$.  Put $\theta=q/p$ and $x=t^\theta$.  Since $0\le t\le1$ and
$0<\theta\le1$, one has $x\ge t$ and $x^p=t^q$.  Hence
\begin{align*}
 G_{r+s}(t^q)
 &=G_{r+s}(x^p)\\
 &\ge G_r(x)^pG_s(x)^p\\
 &\ge G_r(t)^pG_s(t)^p\\
 &\ge G_r(t)^qG_s(t)^q,
\end{align*}
where the last inequality uses $G_r(t),G_s(t)\ge1$.  The exponent $q=0$
is admissible by definition.

Finally, let $p_n$ be admissible and $p_n\to p$.  If $p<1/2$, then $p$ is
admissible by downward closure from the admissible exponent $1/2$.  If
$p\ge1/2$, then $t^p$ is jointly continuous in $(t,p)$ on
$[0,1]\times[1/4,\infty)$; in particular, there is no ambiguity at $t=0$.
For each fixed $t\in[0,1]$, both sides of \eqref{eq:block-def-intro} therefore
depend continuously on $p$ in a neighborhood of the limit.  Passing to the
limit shows that $p$ is admissible.  Thus the
admissible set is closed, and its supremum is attained.
\end{proof}

The following geometric-block estimate was proved by the first author.
It is the only result imported from \cite{Hosle2026} that is used in the
proof of the fixed-head theorem.

\begin{proposition}[Geometric-block inequality
{\cite[Theorem~1]{Hosle2026}}]\label{prop:hosle-block}
Let
\[
        q_m:=\frac{\log(2m+1)}{2\log(m+1)}.
\]
If $1\le r,s\le m$, then
\begin{equation}\label{eq:hosle-block}
        G_{r+s}(t^{q_m})\ge G_r(t)^{q_m}G_s(t)^{q_m}
        \qquad(0\le t\le1).
\end{equation}
Consequently,
\begin{equation}\label{eq:hosle-lower}
        p_{r,s}\ge q_{\max\{r,s\}}.
\end{equation}
\end{proposition}

\begin{remark}
For $r=s=m$, equality at $t=1$ shows that $q_m=p_{m,m}$.  The novelty in
Theorem~\ref{thm:endpoint-active} is that for a fixed nonsymmetric pair
$(r,s)$ the endpoint $t=1$ need not be active.
\end{remark}

\section{The endpoint-active criterion}\label{sec:endpoint}

We begin with the determinantal sign principle used in the interpolation
argument.  It is a standard form of the generalized mean-value theorem for
determinants; see, for example, \cite{KarlinStudden1966}.  We include the proof.

\begin{lemma}[Determinantal sign rule]\label{lem:det-sign}
Let $I\subset\R$ be an interval and let $f_0,\ldots,f_n\in C^n(I)$.  Assume
that each initial Wronskian
\[
        W_k(x):=W(f_0,\ldots,f_k)(x),
        \qquad 0\le k\le n,
\]
is nonzero on $I$.  Then every nonzero element of
$\operatorname{span}\{f_0,\ldots,f_n\}$ has at most $n$ distinct zeros in
$I$.  Since $W_n$ is continuous and nonvanishing, its sign is constant on
$I$.  Consequently, for distinct $z_0,\ldots,z_n\in I$ and any $x_*\in I$,
\begin{equation}\label{eq:det-sign-rule}
 \sgn\det(f_j(z_i))_{i,j=0}^n
 =
 \sgn W_n(x_*)\,
 \sgn\prod_{0\le i<j\le n}(z_j-z_i).
\end{equation}
\end{lemma}

\begin{proof}
We first prove the zero-counting assertion by induction on $n$.  The case
$n=0$ is immediate because $W_0=f_0$ never vanishes.  Assume $n\ge1$ and
write
\[
        F=\sum_{j=0}^n c_jf_j.
\]
Since $f_0$ never vanishes, the quotient $F/f_0$ is well-defined on $I$.  If
$F$ has $n+1$ distinct zeros, then $F/f_0$ has the same zeros, and Rolle's
theorem shows that $(F/f_0)'$ has at least $n$ distinct zeros.  Put
\[
        g_j:=\left(\frac{f_j}{f_0}\right)',
        \qquad 1\le j\le n.
\]
For $1\le k\le n$, set $u_0=1$ and $u_j=f_j/f_0$.  The matrix with entries
$(f_0u_j)^{(i)}$, $0\le i,j\le k$, is the product of a lower-triangular
matrix with diagonal entries $f_0$ and the matrix with entries $u_j^{(i)}$.
Consequently,
\[
 W(f_0,f_1,\ldots,f_k)=f_0^{k+1}W(1,u_1,\ldots,u_k).
\]
Expanding the last Wronskian along its first column gives
\[
 W(1,u_1,\ldots,u_k)=W(u_1',\ldots,u_k')=W(g_1,\ldots,g_k),
\]
and hence
\begin{equation}\label{eq:derived-wronskian}
 W(g_1,\ldots,g_k)
 =\frac{W(f_0,\ldots,f_k)}{f_0^{k+1}}.
\end{equation}
Every initial Wronskian of $g_1,\ldots,g_n$ is therefore nonzero.  By the
induction hypothesis, a nonzero element of their span has at most $n-1$
distinct zeros.  Now
\[
        \left(\frac{F}{f_0}\right)'
        =\sum_{j=1}^n c_jg_j.
\]
If this function is nonzero, the $n$ zeros supplied by Rolle's theorem give
a contradiction.  If it is identically zero, then $F/f_0$ is constant, so
$F$ is a constant multiple of the nowhere-vanishing function $f_0$; its
assumed zeros then force $F\equiv0$.  This proves the zero-counting assertion.

We next prove the determinant sign formula.  When $n=0$, it reads
$\sgn f_0(z_0)=\sgn W_0(x_*)$ and is immediate because the sign of $f_0$ is
constant on $I$.  Assume $n\ge1$.  The zero-counting assertion implies that
\[
        D(z_0,\ldots,z_n):=\det(f_j(z_i))_{i,j=0}^n
\]
is nonzero whenever the nodes are distinct: otherwise a nontrivial linear
combination of the columns would vanish at all $n+1$ nodes.  The chamber
\[
        \{(z_0,\ldots,z_n)\in I^{n+1}:z_0<\cdots<z_n\}
\]
is connected, so the continuous function $D$ has constant sign there.
We may assume that $I$ has nonempty interior, since otherwise $I$ is a singleton and the claim is immediate.
Choose an interior point $z\in I$ and, for sufficiently small $\delta>0$,
put $z_i=z+i\delta$.

For distinct nodes $w_0,\ldots,w_n$, let
$f[w_0,\ldots,w_i]$ denote the divided difference of order $i$.  Successive
row operations in the usual divided-difference table give the identity
\begin{equation}\label{eq:divided-difference-det}
 \det(f_j(w_i))_{i,j=0}^n
 =\left(\prod_{0\le i<j\le n}(w_j-w_i)\right)
   \det\bigl(f_j[w_0,\ldots,w_i]\bigr)_{i,j=0}^n.
\end{equation}
Indeed, at stage $k=1,\ldots,n$, and for $i=n,n-1,\ldots,k$, replace
row $i$ by its difference with row $i-1$, divided by $w_i-w_{i-k}$.  After
stage $k$, row $i$ contains the divided difference on
$w_{i-k},\ldots,w_i$.  At the end, row $i$ contains
$f_j[w_0,\ldots,w_i]$.  The product of all divisors is
\[
 \prod_{k=1}^n\prod_{i=k}^n(w_i-w_{i-k})
 =\prod_{0\le i<j\le n}(w_j-w_i),
\]
which proves \eqref{eq:divided-difference-det}.

Apply this identity with $w_i=z+i\delta$.  Repeated Rolle's theorem for
divided differences gives, for each $i$ and $j$, a point
$\xi_{i,j,\delta}\in[z,z+i\delta]$ such that
\[
 f_j[z,z+\delta,\ldots,z+i\delta]
 =\frac{f_j^{(i)}(\xi_{i,j,\delta})}{i!}.
\]
As $\delta\downarrow0$, the point $\xi_{i,j,\delta}$ tends to $z$; continuity
of $f_j^{(i)}$ therefore gives
\[
        f_j[z,z+\delta,\ldots,z+i\delta]
        \longrightarrow \frac{f_j^{(i)}(z)}{i!}.
\]
Consequently,
\[
 \frac{D(z,z+\delta,\ldots,z+n\delta)}
      {\delta^{n(n+1)/2}\prod_{0\le i<j\le n}(j-i)}
 \longrightarrow
 \frac{W_n(z)}{\prod_{k=0}^n k!}.
\]
The denominator is positive for $\delta>0$, so $D$ has the sign of $W_n(z)$
for all sufficiently small $\delta$.  Since its sign is constant on the
ordered chamber, $D$ has the sign of $W_n$ throughout that chamber.  For arbitrarily
ordered distinct nodes, permuting the rows into increasing order multiplies
the determinant by the sign of the same permutation, which is
\[
        \sgn\prod_{0\le i<j\le n}(z_j-z_i).
\]
This proves \eqref{eq:det-sign-rule}.
\end{proof}

\begin{proof}[Proof of Theorem~\ref{thm:endpoint-active}]
Testing \eqref{eq:block-def-intro} at $t=1$ gives
$p_{r,s}\le a_{r,s}$.  We prove first that $p=a_{r,s}$ is admissible when
$a_{r,s}\ge b_{r,s}$.

By symmetry, assume $1\le r\le s$.  Put
\[
        \alpha=r+1,\qquad
        \beta=s+1,\qquad
        \gamma=r+s+1=\alpha+\beta-1,
\]
and write
\[
        a:=a_{r,s}=\frac{\log\gamma}{\log\alpha+\log\beta}.
\]
Because
\[
 \alpha\beta-\gamma=rs>0
 \quad\text{and}\quad
 \gamma^2-\alpha\beta=r^2+rs+s^2+r+s>0,
\]
one has
\begin{equation}\label{eq:a-range}
        \frac12<a<1.
\end{equation}
For real $c>0$ and $0<t<1$, set
\[
        [c]_t:=\frac{1-t^c}{1-t}.
\]
Then
\[
 G_r(t)=[\alpha]_t,
 \qquad G_s(t)=[\beta]_t,
 \qquad G_{r+s}(t)=[\gamma]_t,
\]
and
\[
        [\gamma]_{t^a}=\frac{[a\gamma]_t}{[a]_t}.
\]
Thus the desired inequality at exponent $a$ is equivalent to
\begin{equation}\label{eq:q-number-target}
        [a\gamma]_t\ge [a]_t[\alpha]_t^a[\beta]_t^a.
\end{equation}

Write $t=e^{-\lambda}$ with $\lambda>0$, and define
\[
        g(c):=\log[c]_t,
        \qquad
        h(c):=g'(c)=\frac{\lambda}{e^{\lambda c}-1}.
\]
By the fundamental theorem of calculus, each difference
$g(v)-g(u)$ is $\int_u^v h(c)\,dc$.  Hence the logarithm of the quotient in
\eqref{eq:q-number-target} is
\begin{equation}\label{eq:integral-rho}
 g(a\gamma)-g(a)-ag(\alpha)-ag(\beta)
 =\int_0^\infty h(c)\rho(c)\,dc,
\end{equation}
where
\begin{equation}\label{eq:rho-def}
 \rho(c):=
 \one_{[a,a\gamma]}(c)
 -a\one_{[1,\alpha]}(c)
 -a\one_{[1,\beta]}(c).
\end{equation}
The relevant moments are
\begin{equation}\label{eq:rho-moments-zero}
        \int_0^\infty\rho(c)\,dc=0,
        \qquad
        \int_0^\infty\frac{\rho(c)}{c}\,dc=0,
\end{equation}
and
\begin{align}
 \int_0^\infty c\rho(c)\,dc
 &=\frac{(a\gamma)^2-a^2}{2}
   -a\frac{\alpha^2-1}{2}
   -a\frac{\beta^2-1}{2} \notag\\
 &=\frac{a}{2}
 \left(a(\gamma^2-1)-(\alpha^2+\beta^2-2)\right).
 \label{eq:rho-first-moment}
\end{align}
Indeed,
\[
 \int_0^\infty\rho(c)\,dc
 =a(\gamma-1)-a(\alpha-1)-a(\beta-1)=0
\]
because $\gamma=\alpha+\beta-1$, while
\[
 \int_0^\infty\frac{\rho(c)}c\,dc
 =\log\gamma-a\log\alpha-a\log\beta=0
\]
by the definition of $a$.  Moreover,
\[
 \gamma^2-1=(r+s)(r+s+2),
 \qquad
 \alpha^2+\beta^2-2=r(r+2)+s(s+2),
\]
so the assumption $a\ge b_{r,s}$ is exactly
\begin{equation}\label{eq:positive-c-moment}
        \int_0^\infty c\rho(c)\,dc\ge0.
\end{equation}

We next establish the order of the interpolation nodes.  A direct
simplification gives
\begin{equation}\label{eq:bgamma-alpha}
 b_{r,s}\gamma-\alpha
 =\frac{s(s^2+2s-r^2)}{(r+s)(r+s+2)}>0,
\end{equation}
because $s\ge r\ge1$.  Hence
\begin{equation}\label{eq:node-order}
        0<a<1<\alpha<a\gamma.
\end{equation}
Set
\[
        x_1=1,\qquad x_2=\alpha,\qquad x_3=a\gamma.
\]
Let
\[
        P(c)=A+\frac{B}{c}+Cc
\]
be the unique interpolant satisfying $P(x_i)=h(x_i)$ for $i=1,2,3$, and put
$E=h-P$.

We claim that $C>0$ and
\begin{equation}\label{eq:error-sign-pattern}
 \begin{array}{c|cccc}
 c&(0,x_1)&(x_1,x_2)&(x_2,x_3)&(x_3,\infty)\\ \hline
 \sgn E(c)&+&-&+&-
 \end{array}.
\end{equation}
To prove this, set $z=\lambda c$.  Differentiating
$h(c)=\lambda/(e^z-1)$ gives
\begin{align*}
 h'(c)&=-\frac{\lambda^2e^z}{(e^z-1)^2},\\
 h''(c)&=\frac{\lambda^3e^z(e^z+1)}{(e^z-1)^3},\\
 h'''(c)&=-\frac{\lambda^4e^z(e^{2z}+4e^z+1)}{(e^z-1)^4}.
\end{align*}
Substitution into the two determinants gives
\begin{align}
 W\left(1,\frac1c,h\right)
 &=-\frac{\lambda^2e^z}{c^3(e^z-1)^3}
 \left(z(e^z+1)-2(e^z-1)\right)<0,
 \label{eq:wronskian-three}\\
 W\left(1,\frac1c,c,h\right)
 &=\frac{4\lambda^3e^{2z}}{c^4(e^z-1)^4}
 \left(z(\cosh z+2)-3\sinh z\right)>0.
 \label{eq:wronskian-four}
\end{align}
For the first sign, observe that
\[
        z(e^z+1)>2(e^z-1)
        \quad\Longleftrightarrow\quad
        \frac z2>\tanh\frac z2.
\]
The function $u-\tanh u$ vanishes at $u=0$ and has derivative
$1-\operatorname{sech}^2u=\tanh^2u>0$ for $u>0$, so the inequality is
strict.  For the second, if
\[
        \Phi(z):=z(\cosh z+2)-3\sinh z,
\]
then
\[
        \Phi(0)=\Phi'(0)=\Phi''(0)=0,
        \qquad
        \Phi'''(z)=z\,\sinh z>0
\]
for $z>0$.  Hence $\Phi''$ is strictly increasing from $0$, so
$\Phi''(z)>0$; integrating twice more gives $\Phi'(z)>0$ and then
$\Phi(z)>0$ for every $z>0$.  The remaining initial Wronskians are
\[
 W(1)=1,
 \qquad
 W\left(1,\frac1c\right)=-\frac1{c^2},
 \qquad
 W\left(1,\frac1c,c\right)=-\frac2{c^3}.
\]
Thus Lemma~\ref{lem:det-sign} applies to both systems
$(1,c^{-1},h)$ and $(1,c^{-1},c,h)$.

By Cramer's rule,
\[
 C=
 \frac{
 \det\bigl(1,x_i^{-1},h(x_i)\bigr)_{i=1}^3}
 {
 \det\bigl(1,x_i^{-1},x_i\bigr)_{i=1}^3}.
\]
The numerator is negative by \eqref{eq:wronskian-three} and
Lemma~\ref{lem:det-sign}, while
\begin{equation}\label{eq:denominator-interpolation}
 \det\bigl(1,x_i^{-1},x_i\bigr)_{i=1}^3
 =-
 \frac{(x_2-x_1)(x_3-x_1)(x_3-x_2)}{x_1x_2x_3}<0.
\end{equation}
Hence $C>0$.

For $c\notin\{x_1,x_2,x_3\}$, define
\[
 \Delta(c):=
 \det\begin{pmatrix}
 1&x_1^{-1}&x_1&h(x_1)\\
 1&x_2^{-1}&x_2&h(x_2)\\
 1&x_3^{-1}&x_3&h(x_3)\\
 1&c^{-1}&c&h(c)
 \end{pmatrix}.
\]
Subtracting the interpolant from the final column yields
\[
        \Delta(c)=E(c)
        \det\bigl(1,x_i^{-1},x_i\bigr)_{i=1}^3.
\]
By \eqref{eq:wronskian-four} and Lemma~\ref{lem:det-sign},
\[
        \sgn\Delta(c)
        =\sgn\bigl((c-x_1)(c-x_2)(c-x_3)\bigr).
\]
Since the determinant in \eqref{eq:denominator-interpolation} is negative,
this proves \eqref{eq:error-sign-pattern}.

We now compare $E$ with $\rho$.  From \eqref{eq:a-range},
\eqref{eq:node-order}, and $\alpha\le\beta$, one has
\[
 \rho>0\text{ on }(a,1),\qquad
 \rho=1-2a<0\text{ on }(1,\alpha),
\]
and $\rho>0$ on $(\alpha,a\gamma)$.  If $a\gamma<\beta$, then additionally
$\rho=-a<0$ on $(a\gamma,\beta)$; outside these intervals $\rho=0$.
Therefore \eqref{eq:error-sign-pattern} gives
\begin{equation}\label{eq:E-rho-positive}
        E(c)\rho(c)\ge0
        \qquad(c>0).
\end{equation}
Using \eqref{eq:rho-moments-zero}, \eqref{eq:positive-c-moment}, and $C>0$,
we obtain
\begin{align*}
 \int_0^\infty h\rho
 &=\int_0^\infty P\rho+\int_0^\infty E\rho\\
 &=C\int_0^\infty c\rho(c)\,dc+\int_0^\infty E(c)\rho(c)\,dc\\
 &\ge0.
\end{align*}
By \eqref{eq:integral-rho}, this proves \eqref{eq:q-number-target} for
$0<t<1$.  The cases $t=0,1$ follow by continuity.  Thus $a$ is admissible,
and $p_{r,s}=a_{r,s}$.

It remains to prove necessity.  Suppose $a:=a_{r,s}<b_{r,s}$, and define
\[
 F(u):=
 \log G_{r+s}(e^{-au})
 -a\log G_r(e^{-u})
 -a\log G_s(e^{-u}).
\]
Using
\begin{equation}\label{eq:logG-expansion}
 \log G_N(e^{-x})
 =\log(N+1)-\frac N2x+\frac{N(N+2)}{24}x^2+O(x^4)
 \qquad(x\to0),
\end{equation}
we find
\begin{equation}\label{eq:F-expansion}
 F(u)=\frac{a}{24}
 \left(a(r+s)(r+s+2)-r(r+2)-s(s+2)\right)u^2
 +O(u^4).
\end{equation}
The constant term vanishes by the definition of $a$, and the linear terms
cancel identically.  Since $a<b_{r,s}$, the quadratic coefficient is
negative.  Hence $F(u)<0$ for all sufficiently small $u>0$, so exponent $a$
fails for $t=e^{-u}$ close to $1$.  Lemma~\ref{lem:basic-p} shows that the
maximal admissible exponent is attained; therefore $p_{r,s}<a_{r,s}$.

For completeness,
\[
 G_N(e^{-x})
 =e^{-Nx/2}(N+1)
   \frac{\sinh((N+1)x/2)}{(N+1)x/2}
   \left(\frac{\sinh(x/2)}{x/2}\right)^{-1}.
\]
Since
\[
        \log\frac{\sinh z}{z}=\frac{z^2}{6}+O(z^4)
        \qquad(z\to0),
\]
taking logarithms gives
\begin{align*}
 \log G_N(e^{-x})
 &=\log(N+1)-\frac N2x
   +\frac{(N+1)^2-1}{24}x^2+O(x^4)\\
 &=\log(N+1)-\frac N2x
   +\frac{N(N+2)}{24}x^2+O(x^4),
\end{align*}
which is \eqref{eq:logG-expansion}.
\end{proof}

\begin{remark}
The quantity $a_{r,s}$ is the zeroth-order obstruction at $t=1$, while
$b_{r,s}$ is the second-order obstruction.  Theorem~\ref{thm:endpoint-active}
says that no higher-order obstruction appears in the endpoint-active regime.
When $a_{r,s}<b_{r,s}$, the theorem does not give a closed formula for
$p_{r,s}$; it shows only that the endpoint upper bound
$p_{r,s}\le a_{r,s}$ is not sharp.
\end{remark}

\section{The fixed-head constant}\label{sec:fixed-head}

The proof uses a transfer principle that increases the length of the second
block.

\begin{lemma}[One-step transfer]\label{lem:transfer}
Fix $r,s\in\N$ and $0<p\le1$.  Suppose $p$ is admissible for $(r,s)$ and
\begin{equation}\label{eq:transfer-condition}
        p\le\frac{s+1}{r+s+1}.
\end{equation}
Then $p$ is admissible for $(r,s+1)$.
\end{lemma}

\begin{proof}
It is enough to prove
\begin{equation}\label{eq:ratio-transfer}
 \frac{G_{r+s+1}(t^p)}{G_{r+s}(t^p)}
 \ge
 \left(\frac{G_{s+1}(t)}{G_s(t)}\right)^p
 \qquad(0\le t\le1).
\end{equation}
We prove this for $0<t<1$; the endpoints follow by continuity.  Concavity of
$x\mapsto x^p$ gives
\begin{equation}\label{eq:right-ratio-bound}
 \left(\frac{G_{s+1}(t)}{G_s(t)}\right)^p
 =\left(1+\frac{t^{s+1}}{G_s(t)}\right)^p
 \le1+p\frac{t^{s+1}}{G_s(t)}.
\end{equation}
On the other hand,
\[
 \frac{G_{r+s+1}(t^p)}{G_{r+s}(t^p)}
 =1+\frac{t^{p(r+s+1)}}{G_{r+s}(t^p)}.
\]
Condition \eqref{eq:transfer-condition} implies
$t^{p(r+s+1)}\ge t^{s+1}$.  Also, with
$[x]_t=(1-t^x)/(1-t)$,
\[
        G_{r+s}(t^p)=\frac{[p(r+s+1)]_t}{[p]_t}.
\]
Since
\[
        \frac{\partial}{\partial x}[x]_t
        =-\frac{t^x\log t}{1-t}>0,
\]
the map $x\mapsto[x]_t$ is increasing.  Together with
\begin{equation}\label{eq:p-number-lower}
        [p]_t\ge p,
\end{equation}
we obtain
\[
 pG_{r+s}(t^p)
 \le[p(r+s+1)]_t
 \le[s+1]_t
 =G_s(t).
\]
Here \eqref{eq:p-number-lower} follows from the tangent-line inequality
$t^p\le1-p(1-t)$.  Therefore
\[
 \frac{t^{p(r+s+1)}}{G_{r+s}(t^p)}
 \ge p\frac{t^{s+1}}{G_s(t)},
\]
which, together with \eqref{eq:right-ratio-bound}, proves
\eqref{eq:ratio-transfer}.  Multiplying \eqref{eq:ratio-transfer} by the
admissible inequality for $(r,s)$ gives the claim for $(r,s+1)$.
\end{proof}

\begin{proof}[Proof of Theorem~\ref{thm:fixed-head}]
Fix $r\in\N$ and put
\[
        m_r:=\min_{s\ge r}a_{r,s}.
\]
The minimum exists because $a_{r,s}\to1$ as $s\to\infty$, whereas
$0<a_{r,r}<1$.  Indeed, choose $S$ so large that
\[
        a_{r,s}>\frac{1+a_{r,r}}2>a_{r,r}
        \qquad(s\ge S).
\]
Hence the minimum is attained among the finitely many integers
$r\le s<S$.  In particular,
\begin{equation}\label{eq:mr-range}
        0<m_r\le a_{r,r}<1.
\end{equation}
We first prove
\begin{equation}\label{eq:uniform-lower-mr}
        p_{r,s}\ge m_r
        \qquad(s\ge1).
\end{equation}

For $1\le s\le r$, Proposition~\ref{prop:hosle-block} with ambient parameter
$r$ gives
\[
        p_{r,s}\ge q_r=a_{r,r}\ge m_r.
\]
We now treat $s\ge r$.  Set
\[
        c_s:=\frac{s+1}{r+s+1}.
\]
Two direct calculations give
\begin{align}
 c_s-b_{r,s}
 &=\frac{r(s^2-r^2-2r)}{(r+s)(r+s+1)(r+s+2)}\ge0
 &&(s\ge r+1),
 \label{eq:cs-b}\\
 c_{s-1}-b_{r,s}
 &=\frac{r(s-r-2)}{(r+s)(r+s+2)}\ge0
 &&(s\ge r+2).
 \label{eq:cprev-b}
\end{align}

We also need that $(r,r+1)$ is endpoint-active.  We claim
\begin{equation}\label{eq:arr1-c}
        a_{r,r+1}\ge c_{r+1}.
\end{equation}
This is equivalent to
\[
 2(r+1)\log(2r+2)
 \ge(r+2)\log((r+1)(r+2)).
\]
For real $x\ge1$, define
\[
 H(x):=2(x+1)\log(2x+2)
 -(x+2)\log((x+1)(x+2)).
\]
Then
\[
 H(1)=\log\frac{4^4}{6^3}>0
\]
and
\[
 H'(x)=\log\frac{4(x+1)}{x+2}-\frac1{x+1}
 \ge\log2-\frac12>0.
\]
Thus $H(r)>0$, proving \eqref{eq:arr1-c}.  Combining it with
\eqref{eq:cs-b} at $s=r+1$ gives
$a_{r,r+1}\ge b_{r,r+1}$, and hence Theorem~\ref{thm:endpoint-active} yields
\begin{equation}\label{eq:rr1-active}
        p_{r,r+1}=a_{r,r+1}.
\end{equation}

Choose
\[
        s_0:=\min\{s\ge r+1:c_s\ge m_r\}.
\]
This exists because $c_s\nearrow1$ and $m_r<1$.  We claim first that $m_r$
is admissible for every $r\le s\le s_0$.  For $s=r$, this follows from
$p_{r,r}=a_{r,r}\ge m_r$, and for $s=r+1$ from
\eqref{eq:rr1-active}.  If $r+2\le s\le s_0$, then the minimality of $s_0$
gives $c_{s-1}<m_r$.  Since $a_{r,s}\ge m_r$, equation
\eqref{eq:cprev-b} gives
\[
        a_{r,s}\ge m_r>c_{s-1}\ge b_{r,s}.
\]
Theorem~\ref{thm:endpoint-active} therefore implies
$p_{r,s}=a_{r,s}\ge m_r$.

Finally, $c_s$ is increasing and $c_{s_0}\ge m_r$.  The range
\eqref{eq:mr-range} verifies the hypothesis $0<p\le1$ in
Lemma~\ref{lem:transfer}.  Starting at $s_0$ and applying that lemma
repeatedly with $p=m_r$ shows that $m_r$ is
admissible for every $s\ge s_0$.  This proves \eqref{eq:uniform-lower-mr}.

Let $s_*$ attain the minimum defining $m_r$.  The endpoint bound gives
$p_{r,s_*}\le a_{r,s_*}=m_r$, while \eqref{eq:uniform-lower-mr} gives the
reverse inequality.  Hence
\[
        C_r=\inf_{s\ge1}p_{r,s}=m_r,
\]
which proves \eqref{eq:fixed-head-identity-intro}.

It remains to locate the integer minimizer.  For real $\sigma\ge r$, define
\[
 A_r(\sigma):=
 \frac{\log(r+\sigma+1)}{\log(r+1)+\log(\sigma+1)}.
\]
A direct differentiation gives
\begin{equation}\label{eq:Ar-derivative}
 A_r'(\sigma)=
 \frac{\Phi_r(\sigma)}
 {(\sigma+1)(r+\sigma+1)
  \bigl(\log((r+1)(\sigma+1))\bigr)^2},
\end{equation}
where
\[
 \Phi_r(\sigma):=
 (\sigma+1)\log((r+1)(\sigma+1))
 -(r+\sigma+1)\log(r+\sigma+1).
\]
Moreover,
\begin{equation}\label{eq:Phi-derivative}
 \Phi_r'(\sigma)=
 \log\frac{(r+1)(\sigma+1)}{r+\sigma+1}>0.
\end{equation}
Indeed,
\[
 \frac{(r+1)(\sigma+1)}{r+\sigma+1}>1
 \quad\Longleftrightarrow\quad r\sigma>0,
\]
which holds for $r\ge1$ and $\sigma\ge r$.
We have $\Phi_r(r)<0$.  Indeed,
\begin{align*}
 &(2r+1)\log(2r+1)-2(r+1)\log(r+1)\\
 &\qquad=(2r+1)\log\frac{2r+1}{r+1}-\log(r+1)\\
 &\qquad\ge(2r+1)\log\frac32-\log(r+1)>0,
\end{align*}
where the last expression is positive at $r=1$ and its derivative is
\[
        2\log\frac32-\frac1{r+1}
        \ge2\log\frac32-\frac12>0
        \qquad(r\ge1).
\]
On the other hand, using $\log(1+x)\le x$, we have
\begin{align*}
 \Phi_r(\sigma)
 &=(\sigma+1)\log(r+1)
   +(\sigma+1)\log\frac{\sigma+1}{\sigma+r+1}
   -r\log(\sigma+r+1)\\
 &=(\sigma+1)\log(r+1)
   -(\sigma+1)\log\left(1+\frac r{\sigma+1}\right)
   -r\log(\sigma+r+1)\\
 &\ge(\sigma+1)\log(r+1)
   -r-r\log(\sigma+r+1)
 \longrightarrow+\infty.
\end{align*}
Thus \eqref{eq:Phi-derivative} gives a unique zero
$\sigma_r\in(r,\infty)$, characterized by \eqref{eq:sigma-intro}.
Equation \eqref{eq:Ar-derivative} shows that $A_r$ decreases up to
$\sigma_r$ and increases afterwards.  Hence its minimum over integer
$s\ge r$ is attained at $k_r=\lfloor\sigma_r\rfloor$ or $k_r+1$, proving
\eqref{eq:fixed-head-two-intro}.

Finally, at $\sigma=\sigma_r$, equation \eqref{eq:sigma-intro} gives
\[
 \log(r+\sigma_r+1)
 =\frac{\sigma_r+1}{r+\sigma_r+1}
  \log((r+1)(\sigma_r+1)).
\]
Dividing by the final logarithm proves \eqref{eq:real-min-intro}.
\end{proof}

\begin{proof}[Proof of Corollary~\ref{cor:C1}]
Apply Theorem~\ref{thm:fixed-head} with $r=1$.  In the notation of its
proof,
\[
 \Phi_1(2)=3\log6-4\log4<0,
 \qquad
 \Phi_1(3)=4\log8-5\log5>0.
\]
Since $\Phi_1$ is strictly increasing, $2<\sigma_1<3$, and hence
\[
 C_1=\min\{a_{1,2},a_{1,3}\}.
\]
A direct comparison gives
\[
 a_{1,2}=0.773705\ldots<0.773976\ldots=a_{1,3}.
\]
Therefore
\[
        C_1=a_{1,2}=\frac{\log4}{\log6}=p_0.
\]
The endpoint bound and the definition of $C_1$ now give
\[
        C_1\le p_{1,2}\le a_{1,2}=C_1,
\]
so $p_{1,2}=p_0$ as well.

For $L\ge1$, one has $p_{1,L}\ge C_1=p_0$, so
Lemma~\ref{lem:basic-p} shows that $p_0$ is admissible for $(1,L)$.  This is
exactly \eqref{eq:geometric-tail-intro}.  For $L=0$, the same inequality is
\[
        1+t^{p_0}\ge(1+t)^{p_0},
\]
which follows from subadditivity of $u\mapsto u^{p_0}$, since $0<p_0<1$.
\end{proof}

\section{Limiting value of $p_{1, s}$}

In the previous section, we determined $\inf_{s\ge 1}p_{1, s}$. We now show that $\lim_{s\to \infty}p_{1, s}$ exists and equals $\frac{\log 2}{\log(1+\sqrt{2})}.$

\begin{proof}[Proof of Proposition~\ref{prop:p1s-limit}]
For fixed $0\le t<1$, as $s\to\infty$ we have
\begin{equation*}
    G_{s+1}(t^p)\to \frac{1}{1-t^p}
\end{equation*}
and
\begin{equation*}
    (1+t)G_s(t)\to \frac{1+t}{1-t}.
\end{equation*}
Thus the limiting inequality is
\begin{equation}\label{eq:limiting-tail-ineq}
    \frac{1}{1-t^p}
    \ge
    \left(\frac{1+t}{1-t}\right)^p,
    \qquad 0\le t<1.
\end{equation}
Equivalently,
\begin{equation}\label{eq:limiting-tail-f}
    t^p+\left(\frac{1-t}{1+t}\right)^p\ge1.
\end{equation}

Let
\begin{equation*}
    \varphi(t):=\frac{1-t}{1+t}.
\end{equation*}
The map $\varphi$ is a decreasing involution of $[0,1]$, and its unique
fixed point is
\begin{equation*}
    t_0=\sqrt2-1.
\end{equation*}
Indeed, $\varphi(t)=t$ is equivalent to
\begin{equation*}
    t^2+2t-1=0.
\end{equation*}

Set
\begin{equation*}
    p_\infty:=\frac{\log2}{\log(1+\sqrt2)}.
\end{equation*}
Since
\begin{equation*}
    t_0=\sqrt2-1=(1+\sqrt2)^{-1},
\end{equation*}
we have
\begin{equation*}
    t_0^{p_\infty}=\frac12.
\end{equation*}
Testing \eqref{eq:limiting-tail-f} at $t=t_0$ gives
\begin{equation*}
    2t_0^p\ge1.
\end{equation*}
Hence every admissible exponent for the limiting inequality satisfies
\begin{equation*}
    p\le p_\infty.
\end{equation*}

We now prove that $p_\infty$ is admissible for the limiting inequality.
Define
\begin{equation*}
    f(t):=t^{p_\infty}+\varphi(t)^{p_\infty}.
\end{equation*}
We claim that
\begin{equation}\label{eq:f-lower-one}
    f(t)\ge1
    \qquad(0\le t\le1).
\end{equation}
Since $\varphi$ is an involution, we have
\begin{equation*}
    f(\varphi(t))=f(t).
\end{equation*}
Also
\begin{equation*}
    f(0)=1,\qquad f(1)=1,\qquad f(t_0)=2t_0^{p_\infty}=1.
\end{equation*}

For $0<t<1$, we have
\begin{equation*}
    \frac{1}{p_\infty}f'(t)
    =
    t^{p_\infty-1}
    -
    \frac{2\varphi(t)^{p_\infty-1}}{(1+t)^2}.
\end{equation*}
Thus the sign of $f'(t)$ is the sign of
\begin{equation*}
    \Psi(t):=
    (1-p_{\infty})\log\frac{\varphi(t)}{t}
    +
    2\log(1+t)
    -
    \log2.
\end{equation*}
Indeed, $f'(t)>0$ is equivalent to
\begin{equation*}
    \left(\frac{\varphi(t)}{t}\right)^{1-p_{\infty}}(1+t)^2>2,
\end{equation*}
which is exactly $\Psi(t)>0$.

A direct calculation gives
\begin{equation}\label{eq:Psi-derivative}
    \Psi'(t)
    =
    \frac{-(1-p_\infty)+2p_\infty t-(1+p_\infty)t^2}
    {t(1-t)(1+t)}.
\end{equation}
The numerator in \eqref{eq:Psi-derivative} is a quadratic polynomial, so
$\Psi'$ has at most two zeros. Hence, by Rolle's theorem, $\Psi$ has at
most three zeros in $(0,1)$.

We have
\begin{equation*}
    \Psi(0^+)=+\infty,
    \qquad
    \Psi(1^-)=-\infty,
\end{equation*}
and
\begin{equation*}
    \Psi(t_0)=0,
\end{equation*}
because $\varphi(t_0)=t_0$ and $1+t_0=\sqrt2$. Moreover,
\begin{equation*}
    \Psi'(t_0)>0.
\end{equation*}
To see this, using $t_0^2+2t_0-1=0$, the numerator in
\eqref{eq:Psi-derivative} at $t=t_0$ equals
\begin{equation*}
    2\bigl(t_0(1+2p_\infty)-1\bigr).
\end{equation*}
Thus $\Psi'(t_0)>0$ is equivalent to
\begin{equation*}
    p_\infty>\frac1{\sqrt2},
\end{equation*} which is true.

Therefore $\Psi$ crosses from negative to positive at $t=t_0$. Since
$\Psi(0^+)=+\infty$, $\Psi(1^-)=-\infty$, and $\Psi$ has at most three
zeros, its sign pattern is
\begin{equation*}
    +,\ -,\ +,\ -
\end{equation*}
on the four intervals determined by its three zeros, with $t_0$ as the
middle zero. Consequently $f'$ has the same sign pattern. Thus $f$
increases from $f(0)=1$ to a local maximum, decreases to
$f(t_0)=1$, increases to a local maximum, and then decreases to
$f(1)=1$. Hence \eqref{eq:f-lower-one} holds. This proves that
$p_\infty$ is the optimal exponent for the limiting inequality
\eqref{eq:limiting-tail-ineq}.

It remains to pass from the finite inequalities to the limiting one.
First let $p>p_\infty$. At $t=t_0$ we have
\begin{equation*}
    2t_0^p<1,
\end{equation*}
so the limiting inequality fails at $t_0$. Since
\begin{equation*}
    G_{s+1}(t_0^p)\to \frac{1}{1-t_0^p}
\end{equation*}
and
\begin{equation*}
    (1+t_0)G_s(t_0)\to \frac{1+t_0}{1-t_0},
\end{equation*}
the finite inequality \eqref{eq:p1s-finite-ineq} also fails at $t=t_0$
for all sufficiently large $s$. Therefore
\begin{equation*}
    \limsup_{s\to\infty}p_{1,s}\le p_\infty.
\end{equation*}

Conversely, fix $0<p<p_\infty$. Then the limiting inequality is strict
for every $0<t<1$:
\begin{equation*}
    t^p+\varphi(t)^p>1.
\end{equation*}
Equivalently,
\begin{equation*}
    \frac{1}{1-t^p}>
    \left(\frac{1+t}{1-t}\right)^p
    \qquad(0<t<1).
\end{equation*}

We prove that $p$ is admissible for \eqref{eq:p1s-finite-ineq} for all
large $s$. Choose $\delta\in(0,1/2)$ small enough so that
\begin{equation}\label{eq:near-zero-choice}
    1+t^p
    \ge
    \left(\frac{1+t}{1-t}\right)^p
    \qquad(0\le t\le\delta),
\end{equation}
and also so that
\begin{equation}\label{eq:near-one-delta-choice}
    \left(\frac1{2\delta}\right)^{1-p}\ge 2^p.
\end{equation}
Such a choice of $\delta$ is possible: as $t\downarrow0$,
\[
    \left(\frac{1+t}{1-t}\right)^p=1+2pt+O(t^2),
\]
whereas $t^p/t=t^{p-1}\to\infty$, so the term $t^p$ dominates the linear error term and hence
\[
    1+t^p\ge \left(\frac{1+t}{1-t}\right)^p
\]
for all sufficiently small $t$. Decreasing $\delta$ further if necessary, we may also ensure
\[
    \left(\frac1{2\delta}\right)^{1-p}\ge 2^p,
\]
since $1-p>0$.

For $0\le t\le\delta$, we have
\begin{equation*}
    G_{s+1}(t^p)\ge 1+t^p
\end{equation*}
and
\begin{equation*}
    (1+t)G_s(t)\le \frac{1+t}{1-t}.
\end{equation*}
Thus \eqref{eq:near-zero-choice} implies the finite inequality on
$[0,\delta]$, uniformly in $s$.

On the compact interval $[\delta,1-\delta]$, the convergence
\begin{equation*}
    G_{s+1}(t^p)\to \frac{1}{1-t^p}
\end{equation*}
and
\begin{equation*}
    (1+t)G_s(t)\to \frac{1+t}{1-t}
\end{equation*}
is uniform. Since the limiting inequality is strict on this compact
interval, the finite inequality holds on $[\delta,1-\delta]$ for all
sufficiently large $s$.

It remains to handle $1-\delta\le t\le1$. Since $p<1$, we have
$t^p\ge t$ for $0\le t\le1$, and hence
\begin{equation*}
    G_{s+1}(t^p)\ge G_s(t).
\end{equation*}
Also
\begin{equation*}
    \big((1+t)G_s(t)\big)^p\le \big(2G_s(t)\big)^p.
\end{equation*}
Therefore it is enough to show
\begin{equation*}
    G_s(t)\ge \big(2G_s(t)\big)^p,
\end{equation*}
or equivalently
\begin{equation*}
    G_s(t)^{1-p}\ge 2^p.
\end{equation*}
For $t\ge1-\delta$,
\begin{equation*}
    G_s(t)\ge G_s(1-\delta).
\end{equation*}
Since
\begin{equation*}
    G_s(1-\delta)\to \frac1{\delta}
\end{equation*}
as $s\to\infty$, for all sufficiently large $s$ we have
\begin{equation*}
    G_s(1-\delta)\ge \frac1{2\delta}.
\end{equation*}
By \eqref{eq:near-one-delta-choice}, this gives
\begin{equation*}
    G_s(t)^{1-p}\ge 2^p
\end{equation*}
for all $1-\delta\le t\le1$ and all sufficiently large $s$. Hence the
finite inequality holds on $[1-\delta,1]$ for all sufficiently large
$s$.

Combining the three regions, $p$ is admissible for $p_{1,s}$ for all
sufficiently large $s$. Therefore
\begin{equation*}
    \liminf_{s\to\infty}p_{1,s}\ge p.
\end{equation*}
Since this holds for every $0<p<p_\infty$, we get
\begin{equation*}
    \liminf_{s\to\infty}p_{1,s}\ge p_\infty.
\end{equation*}
Together with the upper bound on the limsup, this proves
\begin{equation*}
    \lim_{s\to\infty}p_{1,s}=p_\infty
    =
    \frac{\log2}{\log(1+\sqrt2)}.
\end{equation*}
\end{proof}

\section{Consequences for max-convolution and sumsets}\label{sec:applications}

\subsection{The uniform two-slice inequality}\label{sec:two-slice}

Corollary~\ref{cor:C1} supplies the geometric-block inequality for every
support length.  We now reduce arbitrary nonincreasing two-slice data to that
case.  The max-tie idea was introduced in \cite{BIKM2025}.  More specifically,
the reduction below to adjacent ratios in $\{1,t\}$ and the comparison with a
geometric vector by majorization are adapted from the proof of
\cite[Theorem~8, Sections~5--6]{Hosle2026}.  We include the full specialized
argument.  The new ingredient is Corollary~\ref{cor:C1}, which supplies the
uniform exponent $p_0$ in place of $q_m$.

\begin{lemma}[Max-tie reduction]\label{lem:max-tie}
Fix $0<p<1$, $m\ge1$, and $0<t\le1$.  On
\[
 \Delta_m:=\left\{y_0\ge\cdots\ge y_m\ge0:
                    \sum_{j=0}^m y_j=1\right\},
\]
consider
\[
 \mathcal F_p(y):=
 y_0^p+\sum_{j=1}^m\max\{y_j,ty_{j-1}\}^p+(ty_m)^p-(1+t)^p.
\]
The minimum of $\mathcal F_p$ on $\Delta_m$ is attained at a point with the
following property: if $N=\max\{j:y_j>0\}$, then
\begin{equation}\label{eq:max-tie-ratios}
        \frac{y_j}{y_{j-1}}\in\{1,t\},
        \qquad 1\le j\le N.
\end{equation}
\end{lemma}

\begin{proof}
Partition $\Delta_m$ into the finitely many closed cells obtained by choosing,
for every $1\le j\le m$, one of the inequalities
\[
        y_j\ge ty_{j-1},
        \qquad\text{or}\qquad
        y_j\le ty_{j-1}.
\]
Ties may belong to both cells.  On each cell every maximum in
$\mathcal F_p$ is represented by a fixed coordinate, so $\mathcal F_p$ is a
sum of nonnegative multiples of the concave functions $y_i^p$, minus a
constant.  It is therefore concave on each cell.  To justify the extreme-point
reduction, let $P$ be one such compact polytope and write any $u\in P$ as a
convex combination $u=\sum_\nu\lambda_\nu v_\nu$ of its vertices.  Concavity
gives
\[
        \mathcal F_p(u)\ge
        \sum_\nu\lambda_\nu\mathcal F_p(v_\nu),
\]
so at least one vertex has value at most $\mathcal F_p(u)$.  Thus the minimum
on every cell, and hence the global minimum over the finitely many cells, is
attained at an extreme point.  Choose such a global minimizer $y$.

Let $N=\max\{j:y_j>0\}$.  Suppose that for some $1\le j\le N$ the ratio
$y_j/y_{j-1}$ belongs to neither $\{1,t\}$.  Then neither adjacent equality
$y_j=y_{j-1}$ nor $y_j=ty_{j-1}$ is active.  Form a graph on
$\{0,\ldots,N\}$ whose edges are the active equalities
$y_i=y_{i-1}$ and $y_i=ty_{i-1}$.  The absent edge between $j-1$ and $j$
splits the graph into at least two components.

Choose one component $S$.  Every index in $\{0,\ldots,N\}$ has
positive coordinate, so both $S$ and its complement have positive total
mass.  Set
\[
 c:=\frac{\sum_{i\in S}y_i}
          {\sum_{i\in\{0,\ldots,N\}\setminus S}y_i}>0
\]
and define
\[
 v_i=
 \begin{cases}
 y_i,&i\in S,\\
 -c y_i,&i\in\{0,\ldots,N\}\setminus S,\\
 0,&i>N.
 \end{cases}
\]
Then $\sum_i v_i=0$.  For $0<\eps<\min\{1,1/c\}$, the perturbations
$y\pm\eps v$ are obtained by multiplying the coordinates in $S$ by
$1\pm\eps$ and those in the complement by $1\mp c\eps$; all these factors
are positive.  Every active equality joins two indices in the same graph
component and is therefore preserved.  Every order or cell inequality that
joins distinct components is strict, for otherwise the corresponding edge
would be active.  Since there are only finitely many such strict inequalities,
continuity shows that all of them remain valid when $\eps>0$ is sufficiently
small.  The coordinates after $N$ remain zero, and the equation
$\sum_i y_i=1$ is preserved because $\sum_i v_i=0$.  Hence both
$y+\eps v$ and $y-\eps v$ are distinct points of the same cell whose average
is $y$, contradicting extremality.
This proves \eqref{eq:max-tie-ratios}.
\end{proof}

\begin{lemma}[Geometric majorization]\label{lem:geometric-majorization}
Let $0<t\le1$ and let
$z_0\ge z_1\ge\cdots\ge z_N>0$ satisfy
\[
        z_j\ge t z_{j-1},\qquad 1\le j\le N.
\]
Put $B=\sum_{j=0}^N z_j$ and
\[
        g_j:=\frac{B t^j}{G_N(t)},\qquad 0\le j\le N.
\]
Then $g=(g_0,\ldots,g_N)$ majorizes $z=(z_0,\ldots,z_N)$; that is,
\[
        \sum_{j=0}^k g_j\ge \sum_{j=0}^k z_j
        \quad(0\le k<N),
        \qquad
        \sum_{j=0}^N g_j=\sum_{j=0}^N z_j.
\]
Moreover,
\begin{equation}\label{eq:last-coordinate-majorization}
        z_N\ge\frac{B t^N}{G_N(t)}.
\end{equation}
\end{lemma}

\begin{proof}
The assertion is immediate when $N=0$.  Assume $N\ge1$ and, for
$0\le k<N$, set $S_k=\sum_{j=0}^k z_j$.  The ratio condition gives
\[
 z_j\le t^{-(k-j)}z_k\quad(j\le k),
 \qquad
 z_j\ge t^{j-k}z_k\quad(j>k).
\]
Consequently,
\[
 S_k\le z_k t^{-k}G_k(t),
 \qquad
 B-S_k\ge z_k t^{-k}\bigl(G_N(t)-G_k(t)\bigr).
\]
Hence
\[
 \frac{S_k}{B-S_k}
 \le\frac{G_k(t)}{G_N(t)-G_k(t)},
\]
and therefore
\[
        S_k\le B\frac{G_k(t)}{G_N(t)}
        =\sum_{j=0}^k g_j.
\]
Since $g$ and $z$ have the same total sum and are both nonincreasing, this is
exactly the statement that $g$ majorizes $z$.  Finally, for every $j\le N$,
\[
        z_j\le t^{-(N-j)}z_N.
\]
Summing over $j$ gives $B\le z_N t^{-N}G_N(t)$, which is
\eqref{eq:last-coordinate-majorization}.
\end{proof}

\begin{proof}[Proof of Theorem~\ref{thm:two-slice}]
If $x_0+x_1=0$ or $\sum_j y_j=0$, both sides vanish.  By homogeneity, assume
\[
        x_0+x_1=1,
        \qquad
        \sum_{j=0}^m y_j=1.
\]
If $x_1=0$, then the left-hand side is $\sum_{j=0}^m y_j^{p_0}$, which is at
least $(\sum_jy_j)^{p_0}=1$ by subadditivity of $u\mapsto u^{p_0}$.  Suppose
now that $x_1>0$.  Since $x_0\ge x_1$, there is a unique $t\in(0,1]$ such
that
\[
        x_0=\frac1{1+t},
        \qquad
        x_1=\frac{t}{1+t}.
\]
For the normalized data, multiplying the left-hand side of
\eqref{eq:two-slice} by $(1+t)^{p_0}$ gives exactly
\begin{equation}\label{eq:normalized-max-identity}
 y_0^{p_0}
 +\sum_{j=1}^m\max\{y_j,ty_{j-1}\}^{p_0}
 +(ty_m)^{p_0}.
\end{equation}
The normalized right-hand side becomes $(1+t)^{p_0}$.  Thus
\eqref{eq:two-slice} is equivalent to $\mathcal F_{p_0}(y)\ge0$.

By Lemma~\ref{lem:max-tie}, choose a minimizer satisfying
\eqref{eq:max-tie-ratios}, and let $N=\max\{j:y_j>0\}$.  For
$1\le j\le N$, the ratio condition gives $y_j\ge ty_{j-1}$, and hence
\[
        \max\{y_j,ty_{j-1}\}=y_j.
\]
If $N<m$, the next term is
$\max\{y_{N+1},ty_N\}=ty_N$, and all later terms vanish; if $N=m$, the same
quantity is the terminal term in \eqref{eq:normalized-max-identity}.  It
therefore remains to prove
\begin{equation}\label{eq:defect-goal}
 \sum_{j=0}^N y_j^{p_0}+(ty_N)^{p_0}
 \ge
 (1+t)^{p_0}\left(\sum_{j=0}^Ny_j\right)^{p_0}.
\end{equation}

Put $B=\sum_{j=0}^N y_j$ and
\[
        g_j:=\frac{B t^j}{G_N(t)},\qquad 0\le j\le N.
\]
The ratios in \eqref{eq:max-tie-ratios} satisfy the hypotheses of
Lemma~\ref{lem:geometric-majorization}.  Thus $g$ majorizes $y$.  Since
$u\mapsto u^{p_0}$ is concave, Karamata's inequality
\cite[Chapter~3]{MarshallOlkinArnold2011} gives
\[
        \sum_{j=0}^N y_j^{p_0}
        \ge\sum_{j=0}^N g_j^{p_0}.
\]
The terminal-coordinate estimate in
\eqref{eq:last-coordinate-majorization} also gives
$(ty_N)^{p_0}\ge(tg_N)^{p_0}$.  Consequently,
\begin{align*}
 \sum_{j=0}^N y_j^{p_0}+(ty_N)^{p_0}
 &\ge \sum_{j=0}^N g_j^{p_0}+(tg_N)^{p_0}\\
 &=\frac{B^{p_0}}{G_N(t)^{p_0}}
   G_{N+1}(t^{p_0}).
\end{align*}
Corollary~\ref{cor:C1} now yields
\[
 \frac{B^{p_0}}{G_N(t)^{p_0}}G_{N+1}(t^{p_0})
 \ge (1+t)^{p_0}B^{p_0},
\]
which is \eqref{eq:defect-goal}.  This proves \eqref{eq:two-slice}.

To prove sharpness for $m\ge2$, suppose the same inequality held uniformly
with an exponent $p$.  Take
\[
        x_0=x_1=1,
        \qquad y_0=y_1=y_2=1,
\]
and set $y_j=0$ for $j\ge3$.  The max-convolution has four nonzero entries,
each equal to $1$, so the inequality becomes
\[
        4\ge 2^p3^p=6^p.
\]
Therefore $p\le\log4/\log6=p_0$, as claimed.
\end{proof}

\subsection{Mixed-alphabet sumsets}\label{sec:sumsets}

For nonnegative finitely supported functions $f,g:\Z^d\to[0,\infty)$,
define their max-convolution by
\[
        (f\star_{\max}g)(z):=\max_{x+y=z}f(x)g(y).
\]

\begin{proposition}[Functional mixed-alphabet inequality]
\label{prop:functional-mixed}
Let $m,d\ge1$.  If $f,g:\Z^d\to[0,\infty)$ are supported on
$\{0,1\}^d$ and $\{0,1,\ldots,m\}^d$, respectively, then
\begin{equation}\label{eq:functional-mixed}
 \sum_{z\in\Z^d}(f\star_{\max}g)(z)^{p_0}
 \ge
 \left(\sum_{x\in\Z^d}f(x)\right)^{p_0}
 \left(\sum_{y\in\Z^d}g(y)\right)^{p_0}.
\end{equation}
\end{proposition}

By the dimension-reduction theorem
\cite[Theorem~2.1]{BIKM2025}, it is enough to prove
\eqref{eq:functional-mixed} in dimension $d=1$.  By the rearrangement
principle \cite[Lemma~2.2]{BIKM2025}, after padding the two-term sequence
with zeros, it is enough to consider both sequences nonincreasing.  The
resulting one-dimensional statement is exactly
Theorem~\ref{thm:two-slice}.

\begin{proof}[Proof of Corollary~\ref{cor:sumset}]
Apply Proposition~\ref{prop:functional-mixed} with
$f=\one_A$ and $g=\one_B$.  Since
\[
        \one_A\star_{\max}\one_B=\one_{A+B},
\]
equation \eqref{eq:functional-mixed} gives
\[
        |A+B|\ge (|A||B|)^{p_0}.
\]

For sharpness when $m\ge2$, take
\[
        A=\{0,1\}^d,
        \qquad B=\{0,1,2\}^d.
\]
Then $A+B=\{0,1,2,3\}^d$, so
\[
        |A+B|=4^d,
        \qquad |A||B|=6^d.
\]
If the conclusion held uniformly with exponent $p$, it would give
$4^d\ge6^{pd}$ and hence $p\le\log4/\log6=p_0$.
\end{proof}

\section*{Acknowledgments}

  J.H. was
supported in part by Simons Foundation Collaboration Grant 601948 DJ.  P.I.
acknowledges partial support from NSF CAREER grant DMS-2152401, NSF grant
DMS-2554183, a Simons Fellowship, and a Humboldt Research Fellowship for
Experienced Researchers.  The authors used computational and language-model
tools during exploratory and editorial stages of this project; all
mathematical statements and proofs in the final manuscript were independently
checked by the authors.

\end{document}